\documentclass[11pt,twoside]{amsart}
\usepackage{amssymb}
\usepackage{graphicx,color}

\usepackage[all]{xy}

\newtheorem{thm}{Theorem}[section]
\newtheorem{prop}[thm]{Proposition}
\newtheorem{lem}[thm]{Lemma}

\newtheorem{defn}[thm]{Definition}

\newtheorem{ex}[thm]{Example}
\newtheorem{rem}[thm]{Remark}

\newcommand{\skipit}[1]{{}}
\newcommand{\prfend}{\hbox to7pt{\hfil}
\par\vskip-\baselineskip\hbox to\hsize
{\hfil\vbox {\hrule width6pt height6pt}}\vskip\baselineskip}

\newcommand{\Z}{\mathbb{Z}}

\newcommand {\PP}{\mathbb{P}}

\newcommand{\cO}{\mathcal{O}}

\DeclareMathOperator{\Image}{Im}

\DeclareMathOperator{\conv}{conv}

\newcommand{\myarrow}[2]{\hbox to #1pt{\hfil$\to$\hfil}{\hskip-#1pt{\raise
10pt\hbox to#1pt{\hfil$\scriptscriptstyle #2$\hfil}}}}

\begin{document}

\title{Smooth monomial Togliatti systems of cubics}
\author[Mateusz Micha{\l}ek]{Mateusz Micha{\l}ek}
\address{Facultat de Matem\`atiques, Department d'Algebra i
Geometria, Gran Via de les Corts Catalanes 585, 08007 Barcelona,
Spain}
\address{Mathematical Institute of the Polish Academy of Sciences, \'{S}niadeckich 8, 00-956
Warszawa, Poland}
\email{wajcha2@poczta.onet.pl}

\author[Rosa M. Mir\'o-Roig]{Rosa M. Mir\'o-Roig}
\address{Facultat de Matem\`atiques, Department d'Algebra i
Geometria, Gran Via de les Corts Catalanes 585, 08007 Barcelona,
Spain}
\email{miro@ub.edu}

\begin{abstract} The goal of this paper is to prove the conjecture stated in \cite{MMO}, extending and correcting a previous conjecture of Ilardi \cite{I}, and classify smooth minimal monomial Togliatti systems of cubics in any dimension.

More precisely,
 we classify all minimal monomial artinian ideals $I\subset k[x_0,\cdots ,x_n]$ generated by cubics, failing the weak Lefschetz property and  whose apolar cubic system $I^{-1}$ defines a smooth toric variety. Equivalently, we classify all minimal monomial artinian ideals $I\subset k[x_0,\cdots ,x_n]$ generated by cubics whose apolar cubic system $I^{-1}$ defines a smooth toric variety
 satisfying at least a Laplace equation of order 2. Our methods relay on combinatorial properties of monomial ideals.
\end{abstract}

\thanks{Acknowledgments:
The first author was supported by the research project IP2011 005071 funded by Polish Financial Means for Science in 2012-2014.
The second  author was partially supported
by MTM2013-45075-P.
\\ {\it 2010 Mathematic Subject Classification.} 13E10,  14M25, 14N05, 14N15, 53A20.}

\maketitle


\markboth{M. Micha{\l}ek, R. M. Mir\'o-Roig}{Smooth monomial Togliatti systems of cubics}

\date{August 2014}

\large

\section{Introduction}

An $n$-dimensional variety $X$ whose $d$-osculating space at a general point $x$ has dimension ${n+d\choose d}
-1-\delta$ is said to satisfy $\delta$ independent Laplace equations of order $d$. The classification of $n$-dimensional varieties satisfying at least one Laplace equation of order $d$ is an open problem that has attracted the attention of a huge number of geometers. The roots of this problem go back to work of Togliatti who in \cite{T1} and \cite{T2} classified  all projections of Veronese surfaces in $\PP^9$ which satisfy at least one  Laplace equation of order 2. He proved: The osculating space $T^2_xX$  at a general point $x$ of the rational surface $X=\overline{\Image (\psi)}$ where $\psi $ is the rational map
$$\hskip -2.3cm \psi:\PP^2\longrightarrow \PP^5 $$ $$ \hskip 2.3cm
(x_0,x_1,x_2)\mapsto (x_0^2x_1,x_0x_1^2,x_0^2x_2,x_0x_2^2,x_1^2x_2,x_1x_2^2)$$
has dimension 4 instead of the expected dimension 5.
\vskip 2mm
This example appears also as one of few smooth toric Legendrian varieties in \cite[Example 3]{Bucz08}.
 Inspired by the famous example by Togliatti, in \cite{MMO}
 Mezzetti, Mir\'{o}-Roig and Ottaviani have recently established a link between projections $X$ of the Veronese variety $V(n,d)$ satisfying a Laplace equation of order $d-1$ and artinian ideals $I\subset k[x_0,x_1,\cdots ,x_n]$ generated by forms of degree $d$ which fail the weak Lefschetz property (see Definition \ref{def of wlp}). It turns out that an artinian ideal $I=(F_1, \cdots ,F_r)\subset k[x_0,x_1,\cdots ,x_n]$ generated by $r\le {n+d-1\choose n-1}$ forms of degree $d$ fails the weak Lefschetz property in degree $d-1$ if and only if the projection $X_{n,d}^{I}$ of the Veronese variety $V(n,d)$ from the linear space $ \langle F_1, \cdots ,F_r \rangle\subset |\cO _{\PP^n}(d)|=k[x_0,x_1,\cdots ,x_n]_d$ satisfies a Laplace equation of order $d-1$. In this case we say that $I$ is a {\em Togliatti system} of forms of degree $d$. We also say that $I$ is a smooth (resp. monomial) Togliatti  system if, in addition, $X_{n,d}^{I}$ is a smooth variety (resp. $I$ can be generated by monomials).

\vskip 2mm
The classification of smooth Togliatti systems is a challenging problem far of being solved and so far only partial results have been achieved. In \cite{T1} and \cite{T2} (see also \cite{BK}) Togliatti gave a complete classification of smooth minimal Togliatti systems of cubics in the case $n=2$.  In \cite[Theorem 4.11]{MMO}, Mezzetti, Mir\'{o}-Roig and Ottaviani classified smooth minimal Togliatti systems of cubics in the case $n=3$ and stated a conjecture for arbitrary $n$ \cite[Remark 6.2]{MMO}. The goal of our work is to prove this conjecture.

\vskip 2mm
Let us briefly outline how this paper is organized. In Section 2, we fix notation/definitions and we collect  the background and basic results needed in the sequel. In particular, we recall a close relationship between a priori two unrelated problems: (1) the existence of artinian ideals which fail the weak Lefschetz property, and (2) the existence of smooth varieties satisfying at least a Laplace equation. Section 3 is the heart of the paper.  Let $P$ be a monomial system of cubics in $k[x_0,x_1,\cdots, x_n]$ defining a smooth toric variety and assume that the apolar cubic system fails the weak Lefschetz property and it is minimal (in the sense of Definition \ref{togliattisystem}). Consider $P$ as a subset of $3\Delta $ where $\Delta $ is the standard $n$-dimensional simplex in the lattice $\Z^{n+1}$. 
We associate to $P$ two graphs $G$ and $G_P$. We  start section 3 with a series of technical Lemmas/Propositions gathering the properties of these graphs. These properties will allow us to conclude that all vertices of $3\Delta$ belong to the lattice spanned by $P$ (Proposition \ref{lem:verticesin}) and $P$ contains $x_i^2x_j$ if and only if $P$ also contains $x_ix_j^2$ (Proposition \ref{prop:graph}). Finally, we derive the full classification of minimal smooth monomial  Togliatti systems of cubics (see Theorem \ref{mainthm}).

\vskip 2mm\noindent {\bf Notation.}
$V(n,d)$ will denote the image of the projective space $\PP^n$ in the $d$-tuple
Veronese embedding $\PP^n\longrightarrow \PP^{{n+d\choose d}-1}$.
$(F_1,\ldots,F_r)$ denotes the ideal generated by $F_1,\ldots,F_r$,
while $\langle F_1,\ldots,F_r \rangle$
denotes the vector space they generate.

\vskip 2mm \noindent {\bf Acknowledgement.} The authors would like to thank the IMUB where the research related to this paper was developed.


\section{Definitions and preliminary results} \label{defs and prelim results}

In this section we recall some standard
terminology and notation
from commutative algebra and algebraic geometry,
as well as some results needed later on.

\vskip 2mm
\noindent
{\bf A. The Weak Lefschetz Property.}
Let $R : = k[x_0,x_1,\dots,x_n]=\oplus _tR_t$ be the graded polynomial ring in $n+1$ variables over  an
algebraically closed field $k$ of characteristic zero. We
consider  a homogeneous ideal $I$ of $R$.
The Hilbert
function $h_{R/I}$ of $R/I$ is defined
by $h_{R/I}(t):=\dim _k(R/I)_{t}$.
Note that the Hilbert function of an artinian $k$-algebra $R/I$ has
finite support and is captured in its {\em h-vector} $\underline{h} =
(h_0,h_1,\dots,h_e)$ where $h_0=1$, $h_{i}=h_{R/I}(i)>0$ and $e$ is the
last index with this property.

\begin{defn}\label{def of wlp}\rm Let $I\subset R$ be a
homogeneous artinian ideal. We will say that  the
standard graded artinian algebra $R/I$
 has the {\em weak Lefschetz property (WLP)}
if there is a linear form $L \in (R/I)_1$ such that, for all
integers $j$, the multiplication map
\[
\times L: (R/I)_{j} \to (R/I)_{j+1}
\]
has maximal rank, i.e.\ it is injective or surjective.
(We will often abuse notation and say that the ideal $I$ has the
WLP.)   In this
case, the linear form $L$ is called a {\em Lefschetz element} of
$R/I$.  If for the general form $L \in (R/I)_1$
and for an integer number $j$
the map $\times L$ has not maximal rank we will say that
the ideal $I$ fails the WLP in degree $j$.\end{defn}

The Lefschetz elements of $R/I$ form a Zariski
open, possibly empty, subset of $(R/I)_1$.   Part of the great
interest in the WLP stems from the fact that its presence puts
severe constraints on the possible Hilbert functions, which can
appear in various disguises (see, e.g.,
\cite{St-faces}). Though many algebras
 are expected to have the
WLP, establishing this property is often rather difficult. For
example, it was shown by R. Stanley \cite{st} and J. Watanabe
\cite{w} that a monomial artinian complete intersection ideal
$I\subset R$ has the WLP. By semicontinuity, it follows that
a {\em general}   artinian complete intersection ideal $I\subset
R$ has the WLP but it is open whether {\em every} artinian complete
intersection of height $\ge 4$
over a field of characteristic zero has the WLP. It is worthwhile to point out that in positive characteristic, there are examples of artinian complete
intersection ideals $I\subset k[x,y,z]$ failing the WLP (see, e.g., Remark 7.10 in \cite{MMN2}).

\begin{ex}\rm (Brenner-Kaid's example \cite{BK}) The ideal $I = (x_0^3,x_1^3,x_2^3,x_0x_1x_2)\subset  k[x_0,x_1,x_2]$ fails to have the WLP,
because for any linear form $L = ax_0 + bx_1 + cx_2$ the multiplication map
$$\times L: (k[x_0,x_1,x_2]/I)_2\longrightarrow  (k[x_0,x_1,x_2]/I)_3$$
is neither injective nor surjective. Indeed, since it is a map between two $k$-vector
spaces of dimension 6, to show the latter assertion it is enough to exhibit a nontrivial
element in its kernel. If $f = a^2x_0^2 + b^2x_1^2 + c^2x_2^2- abx_0x_1- acx_0x_2- bcx_1x_2$, then $f\notin I$  and we easily check that $L\cdot f\in  I$.
\end{ex}

In \cite{MMO}, Mezzetti, Mir\'{o}-Roig, and Ottaviani showed  that  the failure of the
WLP can be used to produce  varieties satisfying at least a  Laplace equation. Let us review the needed concepts from differential geometry in order to state their result.

\vskip 2mm \noindent

{\bf B. Laplace Equations.}
Let $X\subset \PP^N$ be a projective variety of dimension $n$ and
let $x\in X$ be a smooth point. Choose
affine coordinates and a local parametrization $\phi $ around $x$ where
 $x=\phi(0,...,0)$ and the $N$
components of $\phi$ are formal power series.
The {\em $s$-th osculating  space}
$T_x^{(s)}X$ to $X$ at $x$ is the projectivised span of  all
partial derivatives of $\phi$ of order
$\leq s$.
The expected dimension of $T_x^{(s)}X$
is ${{n+s}\choose {s}}-1$, but
in general $\dim T_x^{(s)}X\leq {{n+s}\choose {s}}-1$;
if strict inequality
holds for all smooth points of $X$, and $\dim
T_x^{(s)}X={{n+s}\choose {s}}-1-\delta$ for a general point $x$,
then $X$ is said to satisfy $\delta$ Laplace equations of order $s$.

\begin{rem}\label{bound} \rm It is clear that if $N<{{n+s}\choose {s}}-1$ then $X$
satisfies at least one Laplace equation
of order $s$, but this case is not interesting and
will not be considered in the following.
\end{rem}

\vskip 2mm

Recently Mezzetti, Mir\'{o}-Roig  and Ottaviani \cite{MMO} have established a link between
projections of $n$-dimensional Veronese varieties $V(n,d)$ satisfying a Laplace equation and artinian
graded rings $R/I$ failing the WLP. Let us finish this preliminary section  highlighting the existence of
this relationship between a pure algebraic problem:
the existence of artinian ideals $I\subset R$ generated by homogeneous forms
of degree $d$ and failing the WLP; and a pure geometric problem:
the existence of  projections of the Veronese variety
$V(n,d)\subset \PP^{{n+d\choose d}-1}$ in $X\subset \PP^{N}$
satisfying at least one Laplace equation of order $d-1$.

\vskip 2mm

Let $I$ be an artinian
ideal
generated by $r$
  forms
  $F_1, \cdots ,F_r\in R$ of  degree $d$. Associated to $I_d$ there is a morphism
$$\varphi _{I_d}:\PP^n \longrightarrow \PP^{r-1}.$$
Its image $X_{n,I_d}:={\Image (\varphi _{I_d})}\subset \PP^{r-1}$
is the projection of the $n$-dimensional Veronese variety $V(n,d)$
from the linear system $\langle(I^{-1})_d \rangle\subset | \cO _{\PP^n}(d)|=R_d$ where
 $I^{-1}$ is the ideal generated by the Macaulay inverse system of $I$. Analogously, associated to
  $(I^{-1})_d$ there is a rational map
$$\varphi _{(I^{-1})_d}:\PP^n \dashrightarrow \PP^{{n+d
\choose d}-r-1}.$$ The closure of its image $X_{n,(I^{-1})_d}:=\overline{\Image (
\varphi _{(I^{-1})_d})}\subset \PP^{{n+d\choose d}-r-1}$
is the projection of the $n$-dimensional Veronese variety
$V(n,d)$ from the linear system $\langle F_1,\cdots ,F_r \rangle \subset |
\cO _{\PP^n}(d)|=R_d$.
The varieties  $X_{n,I_d}$ and
$X_{n,(I^{-1})_d}$ are usually called  apolar.

\vskip 2mm
Note that when $I\subset R$ is an artinian monomial ideal the inverse system $I^{-1}_d$
is spanned by the monomials in $R$ not in $I$.

We are now ready to state the main result of this section. We have:

\begin{thm}\label{teathm} Let $I\subset R$ be an artinian
ideal
generated
by $r$ homogeneous polynomials $F_1,...,F_{r}$ of degree $d$.
If
$r\le {n+d-1\choose n-1}$, then
  the following conditions are equivalent:
\begin{itemize}
\item[(1)] The ideal $I$ fails the WLP in degree $d-1$,
\item[(2)] The  homogeneous forms $F_1,...,F_{r}$ become
$k$-linearly dependent on a general hyperplane $H$ of $\PP^n$,
\item[(3)] The $n$-dimensional   variety
$X_{n,(I^{-1})_d}$ satisfies at least one Laplace equation of order $d-1$.
\end{itemize}
\end{thm}

\begin{proof} See \cite[Theorem 3.2]{MMO}.
\end{proof}

\begin{rem} \rm \label{bound1}
Note that, in view of Remark \ref{bound}, the assumption
$r\le {n+d-1\choose
n-1}$ ensures that the Laplace equations obtained in (3)
are not obvious in the sense of Remark \ref{bound}. In the particular case $n=2$, this assumption
gives  $r\leq d+1$.
\end{rem}

\begin{defn} \label{togliattisystem} \rm With notation as above,
we will say that $I^{-1}$ (or $I$) defines a \emph{Togliatti system}
if it satisfies the three equivalent conditions in  Theorem \ref{teathm}.

We will say that $I^{-1}$ (or $I$) is a \emph{monomial Togliatti system} if, in addition, $I^{-1}$ (and hence $I$) can be generated by monomials.

We will say that $I^{-1}$ (or $I$) is a \emph{smooth Togliatti system} if, in addition, $n$-dimensional   variety
$X_{n,(I^{-1})_d}$ is smooth.

 A  monomial  Togliatti system $I^{-1}$ (or $I$) is said to be \emph{minimal} if $I$ is generated by monomials $m_1, \cdots ,m_r$ and there is no proper subset $m_{i_1}, \cdots ,m_{i_{r-1}}$ defining a  monomial  Togliatti system.
\end{defn}

\begin{rem}    \rm
Note that we do not require the larger monomial Togliatti system to be smooth. Indeed, consider the system corresponding to:
$$I^{-1}:=(x_0^2 x_1, x_0 x_1^2, x_0 x_1 x_2, x_0^2 x_3, x_0 x_2 x_3, x_2^2 x_3, x_1 x_2 x_3, x_1^2 x_3,$$$$ x_0 x_1 x_3, x_0^2 x_4, x_0x_1x_4, x_1^2 x_4, x_0 x_2 x_4, x_2^2 x_4, x_1 x_2 x_4).$$
It is a smooth monomial Togliatti system of cubics not contained in any of the Togliatti systems described in Example \ref{ex:partitionexamples}.
However, it is not minimal in the sense of Definition \ref{togliattisystem}, as it can be extended by all cubic monomials not divisible by $x_4^2$ different from $x_0^3,x_1^2,x_2^2,x_3^3$. Such an extended system is a Togliatti system, but not a smooth one.
\end{rem}

The names are in honor to Togliatti who proved that for $n=2$ the only smooth monomial Togliatti system of cubics is $I=(x_0^3,x_1^3,x_2^3,x_0x_1x_2)\subset k[x_0,x_1,x_2]$ (see \cite{T1}, \cite{T2}).
The main goal of our paper is to classify all smooth minimal monomial Togliatti systems of cubics and it will be achieved in the next section.

 We end this section classifying all smooth minimal monomial Togliatti systems of quadrics, the result is elementary and serves as  an example and as a model for what we seek to achieve in the case of cubics.

 \begin{prop}  \label{Togliattisystemofquadrics}  Let $I$ be a minimal smooth monomial Togliatti system of quadrics and $n\ge 3$. Then, there is a bipartition of $n+1$: $n+1=a_1+a_2$ with $n-1\ge a_1\ge a_2\ge 2$, such that, up to permutation of the coordinates
 $$I=(x_0,\cdots ,x_{a_1-1})^2+(x_{a_1},\cdots ,x_{n})^2.$$
\end{prop}
\begin{proof}
First it is easy to check that these are minimal smooth monomial Togliatti systems. Let us fix a minimal smooth monomial Togliatti system of quadrics $S$. Let $P$ be the apolar system. Consider the graph $G$ with vertices $v_0,\dots,v_n$. An edge $(v_i,v_j)$ belongs to $G$ if and only if $x_ix_j\in P$. We may regard $P$ as a subset of integral points of $2\Delta$, where $\Delta$ is the standard simplex with $n+1$ vertices. By \cite[Proposition 1.1]{P} we know that $P$ consists of all integral points belonging to a hyperplane.

\noindent {\bf Claim:} Each path $(v_1,v_2,v_3,v_4)$ in $G$ is a part of a four cycle.

\begin{proof}[Proof of the Claim]
The four points in $2\Delta$ corresponding to four edges of the cycle are coplanar. Thus if any three of them belong to a hyperplane then so do all four.
\end{proof}

Let us note that each connected component of $G$ is either a complete graph or a complete bipartite graph. Indeed, choose a component $C$. If $C$ does not contain odd cycles, then $C$ is bipartite. By the Claim, the shortest path joining vertices in different parts must be of length one, thus $C$ is a complete bipartite graph. If $C$ contains an odd cycle, then by the Claim each shortest odd cycle must be a triangle. Consider the largest complete subgraph $C'$ of $C$. If we choose a vertex $v$ not in $C'$ but joined to $C$ by an edge, then, by the Claim, $v$ is joined to all vertices in $C'$. Thus, $C=C'$ must be a complete graph in this case. To fix notation, a complete graph with two vertices, will be considered as a complete bipartite graph.

If one of the components of $G$ is a vertex, then for the minimal $I$, $P$ is not smooth apart from the case $n=3$ (the reader may wish to consult criterion for smoothness presented after Theorem \ref{mainthm}). However, the cardinality assumption $|S|\leq {{n+d-1}\choose {n-1}}$ translates for quadrics to $|P|\geq n+1$. This is not satisfied, when $G$ has an isolated vertex and $n=3$. Hence, none of the components of $G$ is a vertex.

Suppose that none of the components of $G$ is a bipartite graph. Then, $G$ is a disjoint union of complete graphs. The corresponding points do not belong to a hyperplane. If two components of $G$ would be bipartite graphs, then this would contradict minimality of the Togliatti system. Hence, we can assume that exactly one component is a bipartite graph and none of the components is an isolated vertex. If $G$ is connected, then $I$ is as in the proposition. Otherwise, by the minimality and smoothness we can assume that $G$ has exactly two components: a complete graph that is a triangle and a complete bipartite graph with one part of cardinality one. However, such a system does not satisfy the cardinality assumption.
\end{proof}
\begin{rem}\rm
In case of quadrics, by our assumption, $P$ has to be contained in a hyperplane. This implies that the associated toric variety $X$ is of dimension $n-1$. This is a degenerate case, as for higher $d$ the points of $P$ will be contained in zeros of a polynomial of degree $d-1$. In particular, if $\dim P<n$, then $P$ is contained in a hyperplane, which contradicts the minimality of the Togliatti system.
\end{rem}

\section{The Main Theorem}

In this section, we will restrict our attention
to the monomial case and we will classify all
minimal monomial   systems of cubics  $I\subset R$ failing WLP in degree two, such that the apolar system defines an $n$-dimensional smooth rational variety satisfying at least a Laplace equation of order 2. In other words, we will classify all smooth minimal monomial Togliatti systems of cubics.

\vskip 2mm
The following assumptions and notation are valid from now on.
Let $P=\{m_1,\cdots ,m_s\}$ be a monomial system of cubics in $k[x_0,\dots,x_n]$, defining a smooth toric variety. Let $S=\{m'_1,\cdots ,m'_r\}$ be the apolar cubic system. Therefore, we have  $r+s={n+3\choose 3}$. We always suppose that $S$ defines an artinian ring, so it contains $x_i^3$, $0\le i \le n$. Since we are interested in smooth toric varieties satisfying at least one non-trivial (in the sense of Remark \ref{bound}) Laplace equation of order 2, we will also assume $r\le {n+2 \choose 3}$ (see Remark \ref{bound1}). We also assume that $S$ is minimal, i.e. $S$ fails the weak Lefschetz property in degree 2 and no proper subset of $S$ generates an artinian ideal failing weak Lefschetz property in degree 2. By \cite[Proposition 1.1]{P}, this is equivalent to assume that there is a hyperquadric $Q$ containing all integral points of $P$ and no integral point of $3\Delta \setminus P$ and the same is true for any other hyperquadric $Q'$ containing all integral points of $P$.
First of all, we want to point
out that for monomial ideals (i.e. the ideals invariant under the natural toric action of $(k^*)^n$ on $k[x_0,\ldots ,x_n]$)
to test the WLP there
is no need to consider a general linear form. In fact, we have

\begin{prop}
   \label{lem-L-element}
Let $I \subset R:=k[x_0,x_1,\cdots , x_n]$ be an artinian monomial ideal.
Then $R/I$ has the WLP if and only if  $x_0+x_1 + \cdots + x_n$
is a Lefschetz element for $R/I$.
\end{prop}

\begin{proof} See \cite[Proposition 2.2]{MMN2}.
\end{proof}

\vskip 4mm
\noindent {\bf The example of the truncated simplex:}
Consider the linear system of cubics
$$P=\{ x_{i}^2x_{j} \}_{0\le i\ne j\le n}.$$
Note that $|P| =n(n+1)$.  Let $\varphi _{P}:\PP^n  \dashrightarrow
\PP^{n(n+1)-1}$ be the rational map associated to $P$.
The closure of its image $X:=\overline{\Image (\varphi _{P})}
\subset
\PP^{n(n+1)-1}$ is (projectively equivalent to)
the projection of the Veronese variety $V(n,3)$ from
the linear subspace $$S:= \langle x_0^3, x_1^3, ...,x_n^3,
\{x_ix_jx_k\}_{0\le i<j<k\le n} \rangle  $$ of $\PP^{{n+3\choose 3}-1}$. $X$ is smooth  and it  satisfies a Laplace equation of order 2.

\vskip 4mm
In \cite{I}, p. 12, G. Ilardi
conjectured that
the above example was the only smooth monomial Togliatti system
of cubics of dimension $n(n+1)-1$. This conjecture has been recently disapproved by
Mezzetti, Mir\'{o}-Roig and Ottaviani in \cite{MMO} who gave the following example.

\begin{ex} \label{firstcounterex} \rm Consider the linear system of cubics
$$P = \{ x_{i}^2x_{j} \}_{0\le i\ne j\le n, \{i,j\}\neq \{0,1\}}\cup\{x_0x_1x_i\}_{2\le i\le n}.$$
Note that $\dim <P> =n^2+2n-
3$.
Let $\varphi _{P}:\PP^n  \dashrightarrow
\PP^{n^2+2n-4}$ be the rational map associated to $P$.
The closure of its image $X:=\overline{\Image (\varphi _{P})}\subset
\PP^{n^2+2n-4}$ is (projectively equivalent to)
the projection of the Veronese variety $V(n,3)$ from
the linear subspace $$S:=\langle x_0^3, x_1^3, ...,x_n^3,x_0^2x_1,x_0x_1^2,
\{x_ix_jx_k\}_{0\le i<j<k\le n, (i,j)\neq (0,1)} \rangle $$ of $ \PP^{{n+3\choose 3}-1}$.
Again  $X$ is smooth and it satisfies a Laplace equation of order 2.

Notice that  $n^2+2n-4=n^2+n-1$ if and only if $n=3$. Hence, for $n=3$, Example \ref{firstcounterex} provides a counterexample to Ilardi's conjecture.
Nevertheless $X$ cannot be further projected without acquiring singularities; therefore
 for $n>3$, this example does not give  a counterexample to Ilardi's conjecture.
 \end{ex}

The following series of examples were also presented in \cite{MMO} and provide
 counterexamples to Ilardi's conjecture for any $n\ge 3$.

\vskip 2mm

\begin{ex}  \label{ex:partitionexamples} \rm Let us consider a partition of $n+1$: $n+1=a_1+a_2+\cdots +a_s$ with $n-1\ge a_1\ge a_2 \ge \cdots \ge a_s\ge 1$ and the monomial ideal $$S=(x_0,\cdots ,x_{a_1-1})^3+\cdots +(x_{n+1-a_s},\cdots ,x_{n})^3+ J$$
where  $$J:=(x_ix_jx_k \mid i<j<k \text{ and } \forall 1\le \lambda \le s \quad \#(\{i,j,k \}\cap \{\sum _{\alpha \le \lambda -1}a_{\alpha}, \cdots ,\sum _{\alpha \le \lambda }a_{\alpha} -1 \})\le 1 ).$$

First of all we observe that $S$ is a monomial artinian ideal generated by
\begin{equation} \label{dimS}\mu _{a_1,\cdots ,a_s}:={a_1+2\choose 3}+\cdots +{a_s+2\choose 3}+\sum _{1\le i<j<h\le s}a_ia_ja_h\end{equation}
cubics. The ideal $S$ fails the WLP in degree 2 since all the vertex points in $\Z^{n+1}$ corresponding to monomials in the apolar system $P$ are contained in the quadric $Q$ of equation
$$Q=2\sum _{i=0}^n x_i^2-5\sum _{0\le i<j\le n}x_ix_j+9\sum _{0\le i<j\le a_1-1}x_ix_j
+9\sum _{a_1\le i<j\le a_1+a_2-1}x_ix_j+\cdots +9\sum _{n+1-a_s\le i<j\le a_s}x_ix_j.$$
 Alternatively, the restriction of all  cubics in $S$ to the hyperplane $x_0+\cdots +x_n$ become $k$-linearly dependent (Proposition \ref{lem-L-element}). Moreover, $$\beta _{a_1,\cdots ,a_s}:=|P|={n+3\choose 3}-\mu _{a_1,\cdots ,a_s}$$ and the closure of the image of the rational map $\varphi _{P}:\PP^n \longrightarrow \PP^{\beta _{a_1,\cdots ,a_s}-1}$ is a smooth variety $X$ of dimension $n$ which can be seen as the projection of $V(n,3)$ from the linear space generated by all cubic monomials in $S$.

It remains to prove that $S$ is minimal. First we easily see that the above quadric $Q$ that contains all integral points in $P$ does not contain
any integral point of $3\Delta \setminus P$.  Let us now check that $Q$ is unique. Assume that there is another one $$Q'=\sum _{i=0}^n \mu _ix_i^2+\sum _{0\le i<j\le n} \mu _{i,j}x_ix_j.$$  We first consider two indexes  $i<j$ such that $\#(\{i,j \}\cap \{\sum _{\alpha \le \lambda -1}a_{\alpha}, \cdots ,\sum _{\alpha \le \lambda }a_{\alpha} -1\})\le 1 $), $\forall 1\le \lambda \le s$. Then $x_i^2x_j$ and $x_j^2x_i$ belongs to $P$ and we get 
\begin{align*}
4\mu _i+\mu _j+2\mu_{i,j} & = 0  \\ \mu _i+4\mu _j+2\mu_{i,j} & = 0
\end{align*}  
which gives us
\begin{equation} \label{1relation} \mu _i=\mu _j= \frac{-2\mu_{i,j}}{5} .\end{equation}

Now, we  consider two indexes  $i<j$ such that   $\{i,j \}\subset \{\sum _{\alpha \le \lambda -1}a_{\alpha}, \cdots ,\sum _{\alpha \le \lambda }a_{\alpha}-1 \}$ for certain $ 1\le \lambda \le s$.  Then $x_ix_jx_k\in P$ for all $k\notin \{\sum _{\alpha \le \lambda -1}a_{\alpha}, \cdots ,\sum _{\alpha \le \lambda }a_{\alpha}-1 \}$. Therefore, we have
\begin{equation} \label{secondrelation} \mu_i+\mu_j+\mu_k+\mu_{i,j}+\mu_{i,k}+\mu_{j,k}=0.\end{equation}
By (\ref{1relation}), we have
$\mu_i=\mu_k=\mu_j=\frac{-2\mu_{i,k}}{5}=\frac{-2\mu_{j,k}}{5}$.  Substituting in (\ref{secondrelation}) we obtain $\mu_{i,j}=\frac{-4\mu_{i,k}}{5}$ which finishes the proof.

\vskip 2mm So, we have a series  of examples of  smooth monomial Togliatti system
of cubics (see Definition \ref{togliattisystem}) and, if $a_1=n-1$ and $a_2=a_3=1$ or $a_1=a_2=\cdots =a_{n+1}=1$, they have dimension $n(n+1)-1$.
\end{ex}

In \cite[Remark 6.2]{MMO},   it was conjectured that all  smooth monomial Togliatti systems of cubics are obtained by the above procedure. The main goal of our work is to prove this conjecture. In fact, we have got:

\begin{thm} \label{mainthm}
Let $P$ (or its inverse system $S$) be a minimal smooth monomial Togliatti system of cubics. Then, up to a permutation of the coordinates, the pair $(P,S)$ is one of the examples presented in \ref{ex:partitionexamples}. Moreover, $|S|\le {n+1\choose 3}+n+1$ and  if $|S|={n+1\choose 3}+n+1$  then it corresponds to one of the following partitions:
\begin{enumerate}
\item $n+1=(n-1)+1+1$,
\item $n+1=1+1+\dots+1$,
\item $4=2+2$.
\end{enumerate}
\end{thm}

The proof of our main result will follow from a series of technical lemmas/propositions. It will be finished after Lemma \ref{lem:transedges}.
Before we start the proof let us present motivation. There is a combinatorial criterion \cite[Corollary 3.2]{GKZ}, \cite[p.~138]{stks} to check if a subset $P$ of points in a lattice $L$ defines a smooth toric variety. Namely, the associated toric variety is smooth if and only if the following holds:

{\em For every vertex $v$ of the convex hull $\conv P$, let $v_1,\dots,v_k$ be the first lattice points on the edges going from $v$. The vectors $v_1-v,\dots,v_k-v$ form a lattice basis.}

However, in the above criterion, the set of points $P$ should always be regarded in the lattice that it spans. This assumption is not automatically satisfied in our case, as $P$ may span a proper sublattice. It is hard to directly prove that this is not the case. Thus, we start with a weaker statement, Proposition \ref{lem:verticesin}. The first part of the proof is entirely devoted to proving it. The second step is to prove that if $P$ contains a monomial $xy^2$ then it also contains $x^2y$ - Proposition \ref{prop:graph}. Having these two results, the main theorem follows easily.

\vskip 2mm
Let $\Delta$ be the standard $n$-dimensional simplex in the lattice $\Z^{n+1}$.
We may consider $P$ as a subset of $3\Delta$. By choosing a point $Z\in 3\Delta$ we may consider a lattice spanned by $3\Delta$ with $Z$ as the origin. Let $M$ be the sublattice spanned by the points in $P$.

\begin{prop} \label{lem:verticesin} Let $P$ be a set satisfying the hypothesis of Theorem \ref{mainthm}. Then, all vertices of $3\Delta$ belong to $M$.
\end{prop}
\begin{proof}
Suppose $x_0^3\not\in M$. Then on every edge of $3\Delta$ adjacent to $x_0^3$ there may be at most $1$ point belonging to $M$. We say that an edge between $x_0^3$ and $x_i^3$ is of:
\begin{itemize}
\item type a) if $x_0^2x_i\in M$,
\item type b) if $x_0x_i^2\in M$,
\item type c) if it is neither of type a) or b).
\end{itemize}
Let us define a graph $G$. It has $n$ vertices corresponding to edges of $3\Delta$ adjacent to $x_0^3$. Thus the vertices of $G$ are also of type a), b) and c). There is an edge joining vertices $v_i$ and $v_j$ if and only if $x_0x_ix_j\in M$. 

\vskip 2mm The proof of Proposition \ref{lem:verticesin} will be a consequence  of properties of the graph $G$ listed  in the following lemmas.
\begin{lem} \label{lem:forbiden}
Consider the graph $G$. It holds:
\begin{enumerate}
\item There are no edges between two vertices of type a).
\item There are no edges between a vertex of type a) and a vertex of type b).
\item No vertex of type c) is included in any triangle.
\item There is no edge between a vertex of type b) and a vertex of type c).
\end{enumerate}
\end{lem}
\begin{proof}
In any of the mentioned cases $x_0^3$ would belong to the lattice spanned by $P$.
\end{proof}
\begin{lem}\label{lem:induced}
Consider two vertices $v,w$ of type a) and a vertex $u$ of type c). If there is an edge between $u$ and $v$ in $G$, then so is between $u$ and $w$.
\end{lem}
\begin{proof}
Follows by inspection.
\end{proof}

\noindent {\bf Claim 1:} All points in $M$ corresponding to monomials divisible by $x_0$ belong to a hyperplane. In particular, all points of $P$ corresponding to monomials divisible by $x_0$ belong to a hyperplane.

\vskip 2mm
The points of $M$ corresponding to monomials divisible by $x_0$ correspond either to vertices of type a) and b) in $G$ or to edges in $G$. We will construct a hyperplane that contains all these points.

\vskip 2mm

\begin{lem} \label{lem:3cut}
Consider four vertices $v_1,v_2,v_3,v_4$ of any type in $G$. If there is path of length $3$ joining them, then this path is a part of a $4$ cycle.
\end{lem}
\begin{proof}
The points of $M$ corresponding to four edges of a cycle form a $2$-dimensional rhombus and any $3$ vertices of a rhombus generate the fourth one.
\end{proof}
Let $G_c$ be the restriction of $G$ to vertices of type c).
The following is an easy graph theoretic consequence of previous lemmas.

\begin{lem} \label{lem:cliqueorbi}
Every connected component of $G_c$ is  a complete bipartite graph.
\end{lem}
\begin{proof}
Since a graph is bipartite if and only if there is no odd cycle, if the component is not bipartite we could choose the smallest odd cycle in it. By Lemma \ref{lem:3cut} it has to be a triangle, which contradicts Lemma \ref{lem:forbiden} (3).
Hence, the connected component is a bipartite graph. Choose two vertices in different parts. Consider the shortest path joining them. By Lemma \ref{lem:3cut} it has to be of length one. Hence, the bipartite graph is complete.
\end{proof}

\begin{lem} \label{lem:atoonepart}
Consider a complete bipartite graph $C=(A,B)$ that is a connected component of $G_c$. If part $A$ is connected to some vertex of type $A$, then part $B$ is not. Moreover, in such a case all vertices in part $A$ are connected to all vertices of type a).
\end{lem}
\begin{proof}
If a vertex $v$ of type a) is connected to some vertex $w\in C$, then all vertices of type a) are connected to $w$ by Lemma \ref{lem:induced}. Hence, if both parts would be connected to some vertex of type a), they would be connected to the same vertex. This contradicts Lemma \ref{lem:forbiden} (3).

By choosing any $u\in B$ by Lemma \ref{lem:3cut} the vertex $v$ is connected to all vertices in $A$. By Lemma \ref{lem:induced} all vertices of type a) are connected to all vertices in $A$.
\end{proof}
\begin{proof}[Proof of the Claim 1]
Let $a,b,c$ be respectively the number of points of type a), b) and c). Let $c_1$ (resp. $c_2$) be the number of points of type c) that (resp. do not) belong to a connected component connected to a vertex of type a). Hence, $c_1+c_2=c$. Let $D$ be the set of vertices either of type b) or of type c), but not in a connected component connected to a vertex of type a). We have $|D|=b+c_2$. One can consider a $(b+c_2-1)$-dimensional hyperplane $H'$ containing all points corresponding to vertices in $D$ and edges between them, as all these monomials are divisible by $x_0$, but not by $x_0^2$. We will now construct a hyperplane that contains all edges and vertices in the complement of $D$. There is an $(a-1)$-dimensional hyperplane $H_1$ containing all vertices of type a). We now inductively extend $H_1$ adding vertices of type c) and edges between them in such a way that the dimension of $H_1$ is always equal to $a-1$ plus the number of considered vertices of type c).

Fix a connected component $C$ of $G_c$ that is connected to a vertex of type $a)$. By Lemma \ref{lem:cliqueorbi} the graph $C$ is a complete bipartite graph $(A,B)$. By Lemma \ref{lem:atoonepart} all vertices of type a) are connected to all vertices in $A$. Choose $v\in A$. We may extend $H_1$ by a point corresponding to any edge joining $v$ with a vertex of type a). This increases the dimension of $H_1$ by one. Note that all points corresponding to edges joining $v$ with any other vertex of type a) are automatically in this extension, as $H_1$ contains all vertices of type a). In this way we may extend $H_1$ using all vertices in part $A$. Now consider $w\in B$. Extend $H_1$ by the point corresponding to edge $(v,w)$. We claim that all other points corresponding to edges between $w$ and any vertex in $A$ belong to the extension. Indeed, let $v'\in A$ and let $u$ be any vertex of type a). The extension contains the points corresponding to the three edges: $(w,v)$, $(v,u)$, $(u,v')$. Thus, as in the proof of Lemma \ref{lem:3cut} it must also contain the point corresponding to the fourth edge $(w,v')$. Hence we may extend $H_1$ by all vertices in $B$. Proceeding component by component we obtain a hyperplane $H''$ of dimension $a-1+c_1$. Notice that there are no edges between vertices in $D$ and its complement. Hence, the span of $H'$ and $H''$ satisfies the claim.
\end{proof}

We can now finish the proof of Proposition \ref{lem:verticesin}. By the minimality of $S$, by Claim 1, $P$ would have to contain all points in two hyperplanes, apart from vertices of $3\Delta$. One of these hyperplanes contains all monomials not divisible by $x_0$. If $G$ would contain any edge or a vertex of type b), then $x_0^3\in M$. If there is a vertex of type c), then $S$ would not be minimal. Hence, $P$ contains all monomials not divisible by $x_0$ (apart from pure cubes) and all monomials divisible by $x_0^2$. The points divisible by $x_0^2$ form an $(n-1)$-dimensional simplex. Hence, if $P$ defines a smooth variety, each of $n$ of these vertices must have exactly one additional edge. The points not divisible by $x_0$ form also a smooth $(n-1)$-dimensional polytope, but with $n(n-1)$ vertices. Hence, if $P$ defines a smooth toric variety we must have $n(n-1)=n$, hence $n=2$. For $n=2$ we indeed obtain a smooth polytope, contained in two hyperplanes, and spanning a proper sublattice. However, in this case $S$ is not minimal.
\end{proof}

Let us present the definition/construction of a directed graph $G_P$ associated to the set $P$.

\begin{defn} \rm Given a set $P$ satisfying the hypothesis of Theorem \ref{mainthm} we define the graph $G_P$ which has $n+1$  vertices $v_0,\dots,v_n$, they  correspond to the cubics $x_i^3$. Moreover, there is an edge from $v_i$ to $v_j$ if and only if $x_i^2x_j\in P$.
\end{defn}

\begin{ex} \rm (1) Consider $P=\{ x_{i}^2x_{j} \}_{0\le i\ne j\le 2}.$ Then, the directed graph $G_P$ associated to $P$ is:
$$
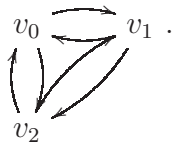
\xymatrix{v_0  \ar @/^/[r] \ar @/^/[d] & v_1 \ar @/^/[dl] \ar @/^/[l]\\ v_2 \ar @/^/[u] \ar @/^/[ur] & } .$$

(2) Consider $P= \{ x_{i}^2x_{j} \}_{0\le i\ne j\le 3, \{i,j\}\neq \{0,1\}}\cup\{x_0x_1x_i\}_{2\le i\le n}$. Then, the directed graph $G_P$ associated to $P$ is:
$$
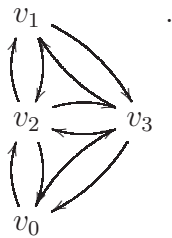
\xymatrix{  v_1 \ar @/^/[dr] \ar @/^/[d] & \\ v_2  \ar  @/^/[u] \ar @/^/[r] \ar @/^/[d] & v_3 \ar @/^/[dl] \ar @/^/[ul]\ar @/^/[l] \\  v_0 \ar @/^/[u] \ar @/^/[ur] & } .$$

\end{ex}
Often we will be using the following, very easy lemma.
\begin{lem}
Consider two polytopes $P_1\subset P_2$. Let $F$ be a face of $P_2$. If $H$ is an $i$-dimensional face of $F\cap P_1$ then $H$ is an $i$-dimensional face of $P_1$.
\end{lem}
\begin{proof}
Choose a hyperplane $L$ that is supporting for $F$. As $P_1\subset P_2$ we have $F\cap P_1=L\cap P_1$. Hence, $F\cap P_1$ is a face of $P_1$ and $H$ is a face of a face, thus a face.
\end{proof}
\begin{lem} \label{lem:noreturn} Let $G_P$ be the graph associated to a set $P$ satisfying the hypothesis of Theorem \ref{mainthm}.
Suppose that the edges $(v_i,v_j),(v_j,v_k),(v_k,v_j)$ belong to $G_P$. Then so does $(v_j,v_i)$.
\end{lem}
\begin{proof}
Consider the face $F$ of $3\Delta$ spanned by $x_i^3,x_j^3,x_k^3$. By Proposition \ref{lem:verticesin} and the assumptions the lattice $M$ contains all points in $F$. If $x_j^2x_i\not\in P$ then $P$ does not define a smooth variety. Indeed, let $\tilde P$ be the convex hull of $P$. Consider the vertex $v:=x_j^2x_k$. On the face $F$ there are two edges adjacent to $v$ going to $x_k^2x_j$ and $x_i^2x_j$. These do not form a basis of the lattice, as the sublattice they generate is of index $2$.
\end{proof}

\begin{lem} \label{lem:out}  Let $G_P$ be the graph associated to a set $P$ satisfying the hypothesis of Theorem \ref{mainthm}.
For every vertex of $G_P$ there is an outgoing edge.
\end{lem}
\begin{proof}
Choose a vertex corresponding to $x_i^3$. If non of the edges are outgoing from it, the system $S$ contains all cubics divisible by $x_i^2$. But this is already a subsystem failing WLP. By minimality, $S$ would have to be equal to this system. The compliment of $S$ in such a case is not a smooth polytope, unless $n=2$. For $n=2$ the cardinality assumption is not satisfied.
\end{proof}

\begin{lem}\label{lem:cycle}  Let $G_P$ be the graph associated to a set $P$ satisfying the hypothesis of Theorem \ref{mainthm}.
Suppose there is an edge $(v_i,v_j)$ in $G_P$ and $(v_j,v_i)$ is not in $G_P$. Then there is a cycle $(v_{a_1},\dots,v_{a_l})$ in $G_P$ of length at least $3$, such that there are no other edges in $G_P$ between the vertices in the cycle.
\end{lem}
\begin{proof}
We may start with an edge $(v_i,v_j)$ and follow the path by Lemma \ref{lem:out}. On such a path there are no returning edges by Lemma \ref{lem:noreturn}. At some point we must obtain a closed cycle $C$. Suppose there is an edge $(w,w')$ between two nonconsecutive vertices of $C$. By Lemma \ref{lem:noreturn} the edge $(w',w)$ does not belong to $G_P$. Hence, we can consider a smaller cycle, containing $(w,w')$. The minimal cycle satisfies the conditions of the lemma.
\end{proof}

\begin{prop} \label{prop:graph}  Let $G_P$ be the graph associated to a set $P$ satisfying the hypothesis of Theorem \ref{mainthm}.
If an edge $(v_i,v_j)$ belongs to $G_P$, then so does $(v_j,v_i)$.
\end{prop}
\begin{proof}
Suppose this is not true. Consider the cycle $C$ from Lemma \ref{lem:cycle}. We may suppose $C=(v_1,\dots,v_l)$. Let $D$ be the face of $3\Delta$ with vertices given by $x_1^3,\dots,x_l^3$. The cubics corresponding to edges in $C$ are vertices of an $(l-1)$-dimensional simplex $\tilde Q\subset D$. However, the cubes $x_i^3$ do not belong to the lattice spanned by $\tilde Q$. By Proposition \ref{lem:verticesin} the intersection $P\cap D$ spans the same lattice as $D$. Hence, $P\cap D$ must contain other points than those in $Q$. By the assumption on the cycle $C$ these must be cubics of the form $m:=x_ix_jx_k$, where $1\leq i<j<k\leq l$. We will prove that $P$ cannot also contain such cubics.

\vskip 2mm
\noindent {\bf Claim 2:} If $m\in P$ then:
\begin{enumerate}
\item $m$ is a vertex of $\tilde P=\conv P$,
\item $\tilde P\cap D$ contains at least $l$ edges adjacent to $m$.
\end{enumerate}
Note that the Claim contradicts smoothness of $P$ as $P\cap D$ is $l-1$ dimensional.

\begin{proof}[Proof of the Claim 2.]
For each edge in the cycle $C$ we will do one of the following:
\begin{enumerate}
\item[(i)] we will assign $1$ to one edge adjacent to $m$ of the convex hull of $P\cap D$,
\item[(ii)] we will assign $1/2$ to two edges adjacent to $m$ of the convex hull of $P\cap D$,
\item[(iii)] for at most one edge, we will assign $1/2$ to one edge adjacent to $m$ of the convex hull of $P\cap D$.
\end{enumerate}
At the end of the procedure we will show that the sum of the numbers assigned to each edge of the convex hull is at most $1$. As we have assigned numbers summing up at least to $l-1/2$, the claim will follow.

As we restrict to variables with indices less or equal to $l$, lying on a cycle, it is more convenient to use cyclic notation modulo $l$. Thus, although as numbers $i<j<k$, from now one we have $k=i-\delta$, where $\delta$ is the length of the directed path from $k$ to $i$.

Let $(v_s,v_{s+1})$ be an edge of the cycle $C$. Let $B$ be the face of $3\Delta$ spanned by $x_i^3,x_j^3,x_k^3,x_s^3,x_{s+1}^3$.
\begin{lem}\label{lem:spantrick}
Let $Q$ be the set containing points corresponding to all edges between the vertices $v_i,v_j,v_k,v_s,v_{s+1}$ (some of these vertices may coincide) and $m=x_ix_jx_k$. Let $L$ be the linear span of $Q$. Suppose that:
\begin{enumerate}
\item $Q$ forms a simplex,
\item $L\cap 3\Delta$ does not contain squarefree monomials different form $m$,
\item $m$ is a vertex of $\tilde P$.
\end{enumerate}
Then either:
\begin{enumerate}
\item there is an edge of $\tilde P$ from $m$ to $x_s^2x_{s+1}$,
\item there are two edges of $\tilde P$ from $m$ to squarefree monomials in variables $x_i,x_j,x_k,x_s,x_{s+1}$.
\end{enumerate}
\end{lem}
\begin{proof}
As $Q$ is a simplex, we have $\dim\tilde P\cap B\geq |Q|-1$. If $\dim\tilde P\cap B=|Q|-1$ then the first conclusion is satisfied, as $\tilde P\cap B$ must be the simplex $Q$. Otherwise, by the smoothness assumption, there are at least $|Q|$ edges adjacent to $m$ in $\tilde P\cap B$. If non of them is adjacent to $x_s^2x_{s+1}$ then at least two of them must be adjacent to squarefree monomials.
\end{proof}
 The Lemma \ref{lem:spantrick} will allow us to associate either $1$ to an edge between $m$ and $x_s^2x_{s+1}$ or $1/2$ to two edges between $m$ and squarefree monomials. The proof of the claim consists of three parts. We often used Polymake software \cite{P} to find all lattice points in the intersection of a hyperplane with a simplex and to check if particular polytopes are smooth.
In the \textbf{first part} we assume that non two of vertices $v_i,v_j,v_k$ are consecutive on the cycle $C$. Equivalently, we exclude the case of monomials $x_ix_jx_k$ where $i,j,k$ do not differ by one (in cyclic notation).

Without loss of generality, we may assume $i\leq s< j$.
We have to consider 3 different  cases:

\vskip 2mm
\noindent \underline{Case 1:} Suppose  $i+3<j$ (i.e. the directed path from $v_i$ to $v_j$ is of length at least $4$).

By considering the 2-dimensional face of $3\Delta$ with vertices $x_i^3,x_j^3,x_k^3$ we see that $m$ is a vertex of $\tilde P$. 

a) Suppose $i+2\leq s< j-2$ or $s=i$ or $s=j-1$. The assumptions of Lemma \ref{lem:spantrick} hold, as $Q$ contains only two points: $m$ and $x_s^2x_{s+1}$. 

b) $s=i+1$.
The assumptions of Lemma \ref{lem:spantrick} hold, as only the three lattice points $m,x_i^2x_{s}, x_{s}^2x_{s+1}$ belong to their linear span intersected with $B$. 

c) $s=j-2$. The same reasoning applies as in point $b)$, as only the three lattice points $m,x_s^2x_{s+1}, x_{s+1}^2x_{j}$ belong to their linear span intersected with $B$.


\vskip 2mm
\noindent \underline{Case 2:} Suppose $j=i+3$.

The only edge of new type is $(v_{i+1},v_{i+2})$. The assumptions of Lemma \ref{lem:spantrick} hold, as only the four lattice points $m,x_i^2x_{i+1}, x_{i+1}^2x_{i+2},x_{i+2}^2x_j$ belong to their linear span intersected with $B$. 

\vskip 2mm
\noindent \underline{Case 3:} Suppose $j=i+2$.

The simplex $B$ is $3$-dimensional. We will consider both edges $(v_i,v_{i+1})$ and $(v_{i+1},v_j)$ simultaneously. If $P\cap B$ is two dimensional, then there are both edges from $m$ to $x_i^2x_{i+1}$ and $x_{i+1}^2x_j$, as there are no other points of $B$ in the linear span of the three points. The only case left is when $P\cap B$ is three dimensional. If there is exactly one edge from $m$ to non-squarefree monomial, then the other two must be  squarefree monomials. In this case we can associate $1$ to the first edge and $1/2$ to the other two. The only case left is when all three squarefree monomials different from $m$, but in $B$ belong to $P$. In this case $P\cap B$ would not be smooth.

This finishes the first part of the proof. We pass to the \textbf{second part}. We assume that exactly two vertices $v_i,v_j,v_k$ are consecutive on the cycle $C$. Equivalently, we suppose $j=i+1$, but $k\neq i+2$, $i-1$ (in cyclic notation). By considering a face spanned by $x_i^3,x_j^3,x_k^3$ we see that $m$ is a vertex of $\tilde P$.
We have to consider 7 different  cases:

\vskip 2mm
\noindent \underline{Case 1:} Suppose the edge $(v_s,v_{s+1})$ is not adjacent to any edge adjacent to $v_i,v_{i+1},v_k$, i.e. either:
\begin{enumerate}
\item $v_s$ is on the path from $v_{i+1}$ to $v_k$ and $i+2<s<k-2$,
\item $v_s$ is on the path from $v_k$ to $v_i$ and $k+1<s<i-2$.
\end{enumerate}
The only non-squarefree monomials in $B\cap P$ are $x_i^2x_{i+1}$ and $x_s^2x_{s+1}$. This two points, together with $m$ span a two dimensional space that intersects $B$ in only one more lattice point: $x_k^2x_{i+1}$, that is not squarefree. Hence, Lemma \ref{lem:spantrick} applies.
Analogously Lemma \ref{lem:spantrick} applies in cases:
\begin{enumerate}
\item $s=i+2$ and $k>i+4$,
\item $s=k-2$ and $k>i+4$,
\item $s=i-2$ and $i>k+3$,
\item $s=i+1$ and $k>i+2$,
\item $s=k-1$ and $k>i+2$,
\item $s=i-1$ and $i>k+2$.
\end{enumerate}
\vskip 2mm
\noindent \underline{Case 2:} $s=k+1$ and $i>k+3$.

There are two squarefree polynomials: $x_{i+1}x_{k+1}x_{k+2}, x_ix_kx_{k+1}$ in the intersection of $B$ with the affine span of $m, x_i^2x_{i+1},x_{k+1}^2x_{k+2},x_{k}^2x_{k+1}$. Extending this system by $1$ or $2$ of these squarefree monomials does not give a smooth $3$-dimensional polytope. Hence $B\cap P$ must be $4$-dimensional. Hence, from $m$ either there is an edge to $x_{k+1}^2x_{k+2}$ or two edges to squarefree monomials.

\vskip 2mm
\noindent \underline{Case 3:} $s=k$ and $i>k+2$.

There is one squarefree monomial: $x_ix_kx_{k+1}$ in the intersection of $B$ with the linear span of $m,x_i^2x_{i+1},x_{k}^2x_{k+1}$. If $B\cap P$ is $2$-dimensional then in both cases there is an edge between $m$ and $x_{k}^2x_{k+1}$. If it is higher dimensional, we can find two edges adjacent to squarefree monomials or to $x_{k}^2x_{k+1}$.

\vskip 2mm
\noindent \underline{Case 4:} $s=i+1,i+2,i+3$ and $k=i+4$.

The polytope $B\cap P$ is $4$-dimensional. The polytope spanned by $$A:=\{m,x_i^2x_{i+1}, x_{i+1}^2x_{i+2},x_{i+2}^2x_{i+3},x_{i+3}^2x_k\}$$ is not smooth. We also know that $x_ix_kx_{i+3}$ is not in $P$, from first part of the proof. The span of $A$ and all squarefree monomials apart from $x_ix_kx_{i+3}$ has an edge between $m$ and $x_{i+1}^2x_{i+2}$ and $m$ and $x_{i+3}^2x_k^2$. Thus for $s=i+1$ and $i+3$ we can associate $1$ to these edges. There must be another edge from $m$ to a non-squarefree monomial or to $x_{i+2}^2x_{i+3}$. We can associate one also to this edge.

\vskip 2mm
\noindent \underline{Case 5:} $s=k+1$ and $i=k+3$.

The system obtained only from edges of $C$ and $m$ without additional squarefree monomials defines a proper sublattice. Hence, we can associate $1/2$ to some edge between a squarefree monomial and $m$. Notice, that this is the only case in which we associate $1/2$ only to one edge.

\vskip 2mm
\noindent \underline{Case 6:} $s=i+1,i+2$ and $k=i+3$.

It is not possible to extend the system to a smooth polytope (the number of vertices has to be even, so there are only 3 possibilities).

\vskip 2mm
\noindent \underline{Case 7:} $s=k,k+1$ and $i=k+2$.

It also is not possible to strictly extend the system to a smooth polytope. If the system is not extended, we can assign $1$ to both edges from $m$ to $x_k^2x_{k+1}, x_{k+1}^2x_i$.

In the last, \textbf{third part} of the proof we suppose $j=i+1$ and $k=i+2$.

We may exclude this case more directly than the previous cases. Suppose that the cycle is of length at least $5$. 
$B\cap P$ contains $m,x_i^2x_{i+1},x_{i+1}^2x_{i+2},x_{i-1}^2x_i$. This is not a smooth polytope, so it must contain some other monomials. Due to previous cases it cannot contain $x_{i-1}x_ix_{i+2}$ and $x_{i-1}x_{i+1}x_{i+2}$. Hence, it may only contain $x_{i-1}x_ix_{i+1}$. This is also not a smooth polytope.

If $C$ is of length $3$ the contradiction is obvious. If $C$ is of length $4$ we may assume $m=x_1x_2x_3$. Then $x_1^2x_2$ must have three edges to: $m,x_2^2x_3$ and $x_4^2x_1$. They do not form a lattice basis.

We have associated in all cases numbers summing up to at least $l$ (or $l-1/2$ in case 5) of the second part). Moreover, each edge to a squarefree monomial was considered at most twice. Thus, the sum of numbers on each edge does not exceed $1$. The Claim 2 follows.
\end{proof}
By the claim we can see that $P\cap D$ cannot be smooth, which gives a contradiction and proofs Proposition \ref{prop:graph}.
\end{proof}
We have seen  that in a graph $G_P$ associated to a set $P$ satisfying the  hypothesis of Theorem \ref{mainthm} either there are no edges between two vertices or there are edges in both directions.

\begin{prop} \label{prop:full}  Let $P$ be the set satisfying the hypothesis of Theorem \ref{mainthm}.
Points of $P$ span the whole ambient lattice $\Z^{n+1}$, i.e.~not a proper sublattice.
\end{prop}
\begin{proof}
We say that a subset $A$ of vertices of $3\Delta$ has property $(*)$ if on a face spanned by $A$ the set $P$ spans the whole lattice.
Suppose $A,B$ are subsets having the property $(*)$. If $A$ and $B$ have non-empty intersection, then $A\cup B$ has also this property. Thus we may define the most coarse partition of vertices of $3\Delta$, such that each part has $(*)$. Our aim is to prove that the partition is trivial. Suppose the partition is given by $A_1,\dots, A_k$. Notice that $P$ cannot contain any monomials that would be of degree $2$ in variables from one part and of degree one in variables from the other part. If $P$ contains monomials in variables from more than one part, then it has to be a product of three variables belonging to different parts. Let us prove that such monomials do not belong to $P$. Suppose that a monomial contains $x_i\in A_i$ for $i=1,2,3$. By Lemma \ref{lem:out} we may find $x_i'\in A_i$ such that $x_i^2x_i',x_i'^2x_i$ belong to $P$. Consider the $5$-dimensional simplex $B$ spanned by $x_i^3,x_i'^3$ for $i=1,2,3$. If $x_1x_2x_3$ belongs to $P$, the convex hull of $P\cap B$ must have the six edges from this monomial. Indeed, for each fixed $i$ there are at least two edges in the simplex spanned by $x_1^3,x_2^3,x_3^3,x_i'^3$. This contradicts the smoothness. Thus $P$ contains only monomials that are in variables belonging to precisely one group. If there is more than one group this contradicts the minimality of the complement of $P$.
\end{proof}

Associated to any set $P$ satisfying the hypothesis of Theorem \ref{mainthm} we construct  a new graph $G_P'$ that will be a (undirected) complement of $G_P$.

\begin{defn} Given a set $P$ satisfying the hypothesis of Theorem \ref{mainthm} we define the graph $G_P'$ as follows:
The vertices  $v_i$ of $G_P'$ correspond to cubes $x_i^3$. An undirected edge $(v_i,v_j)$ belongs to $G_P'$ if and only if both (or equivalently any) cubes $x_i^2x_j, x_j^2x_i$ do not belong to $P$.
\end{defn}

\begin{ex} (1) Consider $P=\{ x_{i}^2x_{j} \}_{0\le i\ne j\le 2}.$ Then, the undirected graph $G_P$ associated to $P$ is:
$$
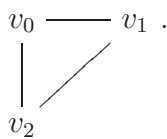
\xymatrix{v_0   \ar@{-} [r] \ar@{-} [d] & v_1 \ar@{-} [dl] \\ v_2  & } .$$

(2) Consider $P= \{ x_{i}^2x_{j} \}_{0\le i\ne j\le 3, \{i,j\}\neq \{0,1\}}\cup\{x_0x_1x_i\}_{2\le i\le n}$. Then, the directed graph $G_P$ associated to $P$ is:
$$
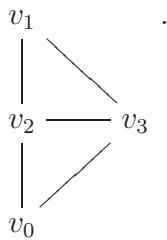
\xymatrix{  v_1 \ar@{-} [dr] \ar@{-} [d] & \\ v_2  \ar@{-} [r] \ar@{-} [d] & v_3 \ar@{-} [dl]  \\  v_0  & } .$$
\end{ex}

\begin{lem}[transitivity of edges]\label{lem:transedges} Let $G_P'$ be the graph associated to a set $P$ satisfying the hypothesis of Theorem \ref{mainthm}.
If $(v_i,v_j),(v_j,v_k)$ are edges of $G_P'$ then so is $(v_i,v_k)$.
\end{lem}
\begin{proof}
 We know that $x_j^3,x_i^3,x_k^3,x_j^2x_i,x_i^2x_j,x_j^2x_k,x_k^2x_j$ do not belong to $P$. Suppose that $(v_i,v_k)$ is not an edge of $G_P'$, i.e. $x_i^2x_k,x_k^2x_i$ belong to $P$. By Lemma \ref{lem:out} we may find $l$ such that $x_j^2x_l\in P$, hence also $x_l^2x_j\in P$ (Proposition \ref{prop:graph}). This would not define a smooth polytope by Proposition \ref{prop:full}. Let $D$ be the $3$-dimensional simplex spanned by $x_i^3,x_j^3,x_k^3,x_l^3$. Consider $D\cap P$. We know that $x_i^2x_k,x_k^2x_i,x_j^2x_l,x_l^2x_j\in D\cap P$. The additional points form a nonempty subset of $$x_i^2x_l,x_l^2x_i,x_k^2x_l,x_l^2x_k,x_jx_ix_k,x_jx_kx_l,x_ix_kx_l,x_jx_ix_l.$$ As the subset is nonempty and the configuration is symmetric we may assume without loss of generality that $x_ix_jx_k\in D\cap P$. By the smoothness there must be other points in $P\cap D$. If all the nonsquarefree monomials belong to $D\cap P$ then it is not smooth. Hence without loss of generality we may assume that $x_i^2x_l,x_l^2x_i\not\in P$. The remaining cases are easily verified using Polymake \cite{P} and do not give smooth polytopes.
\end{proof}

{\textbf{We may now finish the proof of the main theorem.}} By Lemma \ref{lem:transedges} the graph $G_P'$ associated to $P$ is made of disjoint complete graphs. These complete graphs form a partition $A_1,\dots,A_k$ of the vertices of $3\Delta$. First, let us note that $k=1$ would contradict the minimality of $P$. Suppose $|A_1|=n$ and $|A_2|=1$. By smoothness, $P$ contains all the monomials divisible by $x_0$, different from $x_0^3$. If $P$ does not contain other monomials, then the assumption on the cardinality of $P$ is not satisfied. If it contains other monomials, then these must be squarefree monomials that are vertices of $\tilde P$. In such a case some vertex of $P$, corresponding to a monomial divisible by $x_0$ would have an additional edge, contradicting the smoothness assumption.

For all other partitions we may consider the construction presented in Example \ref{ex:partitionexamples}, obtaining $P'$.
Let us note that $P'\subset P$. Indeed, both contain the same non-squarefree monomials. A squarefree monomial belongs to $P'$ if and only if two of the variables belong to one group $A_i$ and the third one to a different $A_j$. By smoothness, such a monomial must also belong to $P$. As we have shown in Example \ref{ex:partitionexamples}, the complement of $P'$ is minimal among smooth monomial Togliatti systems, thus $P=P'$.

Finally, a straightforward computation using the fact that for any partition $n+1=a_1+a_2+\cdots +a_s$, $n-1\ge a_1\ge a_2\ge \cdots \ge a_s$ we have (see (\ref{dimS})): $$\dim S={a_1+2\choose 3}+\cdots +{a_s+2\choose 3}+\sum _{1\le i<j<h\le s}a_ia_ja_h;$$ we conclude that $\dim S \le {n+1\choose 3}+n+1$ and equality holds if and only if $n+1=(n-1)+1+1$ or $n+1=1+1+\cdots +1$ or $4=2+2$.

\begin{rem} \rm If we remove the hypothesis of being smooth the list of monomial Togliatti systems of cubics can be enlarged. For instance,  $$P=(acd, bcd,  a^2c, a^2d,  ac^2, ad^2, b^2c, b^2d, bc^2, bd^2, c^2d, cd^2)$$ is a quasi-smooth  monomial Togliatti systems of cubics. The closure of the image of the rational map $\varphi_P:\PP^3\longrightarrow \PP^{11}$ is
 a 3-fold $X\subset \PP^{11}$  of degree $18$
 and its  normalization  is isomorphic to $\PP^3$ blown up in the line $\{c=d=0\}$ and in the two points
$(0,0,1,0)$ and $(0,0,0,1)$.
\end{rem}


   A further interesting project is the classification of all smooth  monomial Togliatti linear systems of forms of degree $d$ on $\PP^n$ accomplished here for $n\ge 3$  and $d=3$ (see Theorem \ref{mainthm}).

\begin{ex}\rm
For any partition $n+1=a_1+a_2$ with $a_1,a_2>1$ the ideal:
$$I=(x_0,\dots,x_{a_1-1})^d+(x_{a_1},\dots,x_n)^d$$
fails WLP from degree $d-1$ to degree $d$ and is minimal among such systems. Moreover, the apolar system defines a smooth toric variety.
\end{ex}
However, even for $d=4$ and $n=3$ there are other systems.
\begin{ex}\rm
The ideals:
$$I_1=(x_0^4,x_1^4,x_2^4,x_3^4,x_0^2x_1x_2,x_0x_1^2x_2,x_0x_1x_2^2,x_0x_1x_2x_3),$$
$$I_2=(x_0^4,x_1^4,x_2^4,x_3^4,x_0^2x_1x_2,x_0x_1^2x_2,x_0x_1x_2^2,x_0^2x_1x_3,x_0x_1^2x_3,x_0x_1x_3^2),$$
(and all other ideals obtained by permuting the variables) fail WLP from degree $3$ to $4$ and are minimal among such systems. Moreover, respective apolar systems define smooth toric varieties.
\end{ex}
The above examples indicate that further classification is a very challenging open problem - the reader may find more information in \cite{Mezzetti-MiróRoig}.


\end{document}